\documentclass[a4paper,oneside,11pt]{article}
\usepackage[utf8]{inputenc}
\usepackage{lmodern,stackrel}
\usepackage[T1]{fontenc}
\usepackage{verbatim}
\usepackage{textcomp}
\usepackage[english]{babel}
\usepackage[pdftex]{hyperref}
\usepackage[pdftex]{color,graphicx}
\usepackage{amssymb}
\usepackage{amsmath}
\usepackage{upgreek}
\usepackage{graphicx, epstopdf}
\usepackage{dsfont}

\renewcommand{\leq}{\leqslant}
\renewcommand{\geq}{\geqslant}

\usepackage{amsmath,amsfonts,amssymb,amsthm,mathrsfs,stackrel}

\def\build#1_#2^#3{\mathrel{
\mathop{\kern 0pt#1}\limits_{#2}^{#3}}}

\newcommand{\E}{\mathbb{E}}
\newcommand{\V}{\mathbb{V}}

\theoremstyle{plain}
\newtheorem{theorem}{Theorem}

\newtheorem{lemma}{Lemma}

\theoremstyle{definition}

\newtheorem{remark}{Remark}

\newcommand{\Z}{{\mathbb{Z}}}
\newcommand{\N}{{\mathbb{N}}}

\begin{document}
\title{Truncation of long-range percolation models  with square non-summable interactions}
\author{Alberto M. Campos\footnote{Departamento de Matem{\'a}tica, Universidade Federal de Minas Gerais, Av. Ant\^onio
Carlos 6627 C.P. 702 CEP 30123-970 Belo Horizonte-MG, Brazil} and Bernardo N.B. de Lima$^\dagger$}
\date{}
\maketitle


\begin{abstract}We consider some problems related to the truncation question in long-range percolation. Probabilities are given that certain long-range oriented bonds are open; assuming that these probabilities are not summable, we ask if the probability of percolation is positive when we truncate the graph, disallowing bonds of range above a possibly large but finite threshold. This question is still open if the set of vertices is $\Z^2$. We give some conditions under which the answer is affirmative. One of these results generalizes a previous result in [Alves, Hil\'ario, de Lima, Valesin, Journ. Stat. Phys. {\bf 122}, 972 (2017)]. 
\end{abstract}

{\footnotesize Keywords: long-range percolation; truncation; oriented percolation \\
MSC numbers:  60K35, 82B43}

\section{Introduction}

\noindent

Long-range statistical mechanics models are an old topic that have been studied for a long time, e.g., \cite{ACCN, D1, D2} and \cite{FS} for Ising models or \cite{AKN, AN} and \cite{NS} for percolation models.

One of the more intriguing questions in long-range percolation is the so-called {\em truncation question}. In words (we will become more formal later), this question can be stated as follows: consider a translation-invariant long-range percolation model that percolates with positive probability. Is the infinity of range indeed crucial for the occurrence of percolation?

More precisely, let $G=(\V,\E)$ be a transitive graph, where the set of edges $\E$ can be partitioned as $\E=\cup_{n=1}^\infty \E_n$, where $\E_n$ is the set of edges of length $n$. Let $(p_n)_n\in [0,1]$ be a sequence of parameters. Consider on this graph an independent bond percolation model where bonds are open independently, in which each bond $e$ is open with probability $p_{\|e\|}$, where $\|e\|$ is the length of $e$.

Thus, the probability space that describes this model is $(\Omega, {\cal F}, P)$, where $\Omega = \{0,1\}^{\mathbb{E}}$, $\mathcal{F}$ is the canonical product $\sigma$-algebra, and $P = \prod_{e \in \mathbb{E}} \mu_e$, where $\mu_e({\omega}_e = 1) = p_{\|e\|} = 1- \mu_e({\omega}_e = 0)$. An element $\omega \in \Omega$ is called a percolation configuration.

Given a positive integer $K$, define the truncated sequence $(p_n^K)_n$ as

\begin{equation}
p_n^K=\left\{
\begin{array}
[c]{l}%
p_n,\mbox{  if } n \leq K,\\
0,\ \mbox{   if } n>K,
\end{array}\right.\label{eq:truncation}
\end{equation}
and the truncated measure $P^K = \prod_{e \in \mathbb{E}} \mu^K_e$, where $\mu^K_e({\omega}_e = 1) = p^K_{\|e\|} = 1- \mu^K_e({\omega}_e = 0)$.

Then, the truncation question can be restated as: fix a vertex $0\in\V$ (remind that we consider transitive graphs), given a sequence $(p_n)_n$ where $P(0\leftrightarrow\infty)>0$, does there exist a large enough truncation constant $K$ such that $P^K(0\leftrightarrow\infty) >0$ ? (Here we are using the standard notation in percolation where $(0\leftrightarrow\infty)$ means the set of configurations $\omega\in\Omega$ such that there exists an infinite open path starting from the origin.)

Whenever $G=(\V,\E)$ is the $d$-dimensional hypercubic lattice with long range bonds parallel to the coordinate axes, i.e., $\V=\Z^d$ and $\E_n=\{\langle x,x+n.\vec{e}_i \rangle ;x\in\Z^d,i\in\{1,\dots,d\}\}$, where $\vec{e}_i$ is the $i$-th vector in the canonical basis of $\Z^d$; the truncation question can be placed for summable sequences $(p_n)_n$ as well as for non-summable sequences. In the latter case, if $\sum_n p_n= \infty$ by the Borel-Cantelli Lemma, it follows that $P(0\leftrightarrow\infty)=1$.

If $d=1$, it is an exercise to see that the truncation question has a negative answer; when $d\geq 3$, it was shown in \cite{FL} that the truncation question has an affirmative answer. The case $d=2$ is still an open problem and several works tackled this question adding some extra hypotheses upon the sequence $(p_n)_n$ like \cite{Be, FL, FLS, LS, MSV} and \cite{SSV}. 

In some of these results, it is shown that $\lim_{K\rightarrow\infty}P^K(0\leftrightarrow\infty)=1$, which is a little stronger than the truncation question. Indeed in any situation, we have the weak convergence $P^K\Rightarrow P$ when $K\rightarrow\infty$, but the Portmanteau Theorem cannot be applied because the boundary (with respect the product topology) of the event $(0\leftrightarrow\infty)$ has positive probability concerning the measure $P$. 

In Section \ref{normal}, we will give an affirmative answer, for the case $d=2$, with some extra hypotheses that are not included in the papers cited above.

An analogous truncation question can be stated for the $q$-state ferromagnetic Potts model (see Proposition 2 of \cite{FL}) and rephrased as a percolation question, due to the Fortuin-Kastelyn random-cluster representation. It was shown in \cite{F1} and \cite{F2} that the magnetization of the truncated $q$-states long-range Potts model and the probability of percolation on the long-range processes are related by inequality 

$$\mu_{\phi^K_n}^{\beta,s}\geq\frac{1}{q}+\frac{q-1}{q}P^K(0\leftrightarrow\infty),$$
if the $n$-range potential function $(\phi_n)_n$ and the long-range percolation parameters are related by $p_n=\frac{1-\exp (-2\beta\phi_n)}{1+(q-1)\exp (-2\beta\phi_n)}$.

When the sequence $(p_n)_n$ is summable and there is percolation with positive probability, the truncation question can also be stated. The papers \cite{MS} and \cite{Be} are examples where affirmative answers are given. However in \cite{BCC}, a negative answer was given in the context of the Potts model with $q=3$. 

In each section, we will consider the truncation question on a different type of graph. In Section \ref{orient}, we study the truncation question on a special oriented graph, generalizing the result of Theorem 1 of \cite{AHLV}. 

\section{Truncation question on a long-range square lattice}\label{normal}
\noindent

In this section, consider an anisotropic version of the graph $G$. Let $G^{an}=(\Z^2, \E^{an})$ be the graph whose set of bonds is $\E^{an}=\E^v\cup(\cup_n\E^h_n)$, where $\E^v=\{\langle x,x+(0,1) \rangle ;x\in\Z^2\}$ is the set of nearest neighbor vertical bonds and $\E^h_n=\{\langle x,x+(n,0) \rangle ;x\in\Z^2\}$ is the set of horizontal bonds with length $n$. Given the parameters $\delta$ and $(p_n)_n$, each bond $e$ is open, independently, with probability $\delta$ or $p_n$, if $e$ belongs to $\E^v$ or $\E^h_n$, respectively. We continue denoting by $P$ and $P^K$ the non-truncated and truncated at $K$ measures.

Let us remember the H. Kesten result that $p_v+p_h=1$ is the critical curve for independent anisotropic percolation on the ordinary square lattice $\mathbb{L}^2$ (see \cite{Ke} or \cite{Gr}), where vertical and horizontal bonds are open with probabilities $p_v$ and $p_h$, respectively. Indeed, in the next theorem, we will use the following lemma:

\begin{lemma}\label{VH}Consider an independent and anisotropic percolation model on the square lattice $\mathbb{L}^2$ with parameters $p_v$ and $p_h$. Given any $p_v>0$, it holds that $\lim_{p_h\rightarrow 1^-}P_{p_v,p_h}((0,0)\leftrightarrow\infty)=1$.
\end{lemma}

\begin{theorem}\label{class1}Consider the anisotropic percolation model on the graph $G^{an}$ defined above. Given any $\delta>0$, if  the sequence $(p_n)_n$ satisfies $\sum_n p_np_{n+N}=\infty$ for some $N>0$, it holds that $${\displaystyle \lim_{K\to\infty}P^K\{(0,0) \leftrightarrow \infty \} = 1}.$$
\end{theorem}
\begin{proof}Fix $N>0$ such that $\sum_n p_np_{n+N}=\infty$, given any $\epsilon>0$ we can choose integers $M_1$ and $M_2$ satisfying 
$$\exp \left[-\sum_{n=1}^{M_1}p_np_{n+N}\right]< \epsilon \quad\mbox{ and } \quad\exp \left[-\sum_{n=M_1+1}^{M_2}p_np_{n+N}\right]< \epsilon.$$

Given a vertex $(x,y)\in\Z^2$ and $n\in\Z_+$, let us define the following events:
\begin{align*}
&E_{(x,y)}(n)=\{\langle(x,y);(x+n,y)\rangle\mbox{ and }\langle(x+n,y);(x-N,y)\rangle\mbox{ are open}\},\\
&H^-_{(x,y)}=\cup_{n=1}^{M_1} E_{(x,y)}(n)\quad \mbox{ and }\quad H^+_{(x,y)}=\cup_{n=M_1+1}^{M_2} E_{(x,y)}(n).
\end{align*}

Observe that the events $H^{\pm}_.$ use bonds with length at most $M_2+N$; therefore taking $K=M_2+N$ and by definition of $M_1$ and $M_2$, we have that
\begin{align*}
P^K(H^-_{(x,y)})&=1-P^K(\cap_{n=1}^{M_1} (E_{(x,y)}(n))^c)=1-\prod_{n=1}^{M_1}(1-p_np_{n+N})\\
&\geq 1-\exp\left[-\sum_{n=1}^{M_1} p_np_{n+N}\right]>1-\epsilon.
\end{align*}
Analogously, the same bound holds for the probability of $H^+_{(x,y)}$.

Now, we will couple a percolation process on the ordinary square lattice $\mathbb{L}^2$ (with only nearest neighbors non-oriented bonds) in the following manner: given $e=\langle(v_1,v_2);(u_1,u_2)\rangle$ a bond of $\mathbb{L}^2$, define the sequence of events $(X_e)_e$ as follows
\begin{equation*}
\begin{aligned}[c]
 X_e=\begin{cases}
       H^-_{(Nv_1,Nv_2)}, \text{ if }  v_2=u_2,\ u_1-v_1=1 \text{ and $v_1$ is even};\\
       H^+_{(Nv_1,Nv_2)}, \text{ if }  v_2=u_2,\ u_1-v_1=1 \text{ and $v_1$ is odd};\\
       \{\langle (Nv_1,Nv_2);(Nv_1,Nv_2+1)\rangle\mbox{ is open }\}, \text{ if }  v_1=u_1 \mbox{ and }\ u_2-v_2=1. 
	   \end{cases}
\end{aligned}
\end{equation*}

We declare each bond $e$ of $\mathbb{L}^2$ as {\em red} if and only if the event $X_e$ occurs. The appropriate choice of the events $H^-_.$ and $H_.^+$ ensures that the events $(X_e)_e$ are independent. Thus, bonds in $\mathbb{L}^2$ are red following an independent anisotropic bond percolation, where each vertical bond is open with probability $\delta$ and horizontal bonds are open with probability at least $1-\epsilon$. It follows from the definition of $(X_e)_e$ that an infinite red path starting from the origin in $\mathbb{L}^2$ implies an infinite open path starting from the origin in the graph $G^{an}$. By Lemma \ref{VH}, we can conclude that ${\displaystyle \lim_{K\to\infty}P^K\{(0,0) \leftrightarrow \infty \} = 1}.$
\end{proof}

The next lemma, due Kalikow and Weiss (see Theorem 2 of \cite{KW}), is an important fact in the proof of our next result. We state it as we will need later.
\begin{lemma}\label{KW} Consider an independent long-range bond percolation model on the one-dimension graph $(\Z^+,\{\langle i,j\rangle;i,j\in\Z^+\})$ with parameters $(p_n)_n$. If $\sum_n p_n=\infty$ and $\gcd \{n;p_n>0\}=1$, then the random graph on $\Z^+$ formed by open bonds is connected a.s.. Moreover, for all $l\in\Z^+$ it holds that $\lim_{L\rightarrow\infty} P(\{0,1,\dots,l\} \mbox{ are connected in }\{0,1,\dots,L\})=1$.
\end{lemma}

\begin{theorem}\label{class2}Consider the anisotropic percolation model on the graph $G^{an}$. Given any $\delta>0$, if  the sequence $(p_n)_n$ satisfies $\limsup_{N\rightarrow\infty}\sum_n p_np_{n+N}>0$, it holds that $${\displaystyle \lim_{K\to\infty}P^K\{(0,0) \leftrightarrow \infty \} = 1}.$$
\end{theorem}
\begin{proof} Suppose that $\gcd \{n;p_n>0\}=1$ and let $\eta>0$ be such that $\limsup_{N\rightarrow\infty}\sum_n p_np_{n+N}=2\eta$. Given any $\epsilon>0$, choose a large integer $\ell$ satisfying
\begin{equation}\label{l}
\left[1-[1-\delta^2(1-e^{-\eta})]^{\ell}\right]>1-\frac{\epsilon}{3}.
\end{equation}

Given $x\in\Z^2$ and an integer $L>2\ell$, define the following event $$A_x(L)=\{x+\{0,1,\dots,2\ell\}\times\{0\} \mbox{ are connected in }x+\{0,1,\dots,L\}\times\{0\}\}.$$ The hypothesis $\limsup_{N\rightarrow\infty}\sum_n p_np_{n+N}>0$ implies that $\sum_n p_n=\infty$, then by Lemma \ref{KW} we can find a large $L$ such that $P(A_x(L))>1-\epsilon/3$.

Now, choose integers $k>2L$ and $M>L$ such that 
\begin{equation}\label{keM}
\sum_{n=1}^M p_np_{n+k}>\eta.
\end{equation}
Define the events
\begin{align*}R_x^+&=\{\langle x;x+(0,1)\rangle\mbox{ and }\langle x+(k,1);x+(k,2)\rangle\mbox{ are open}\}\\
&\cap\left(\cup_{n=1}^M \{\langle x+(0,1);x+(n+k,1)\rangle\mbox{ and }\langle x+(n+k,1);x+(k,1)\rangle\mbox{ are open}\}\right)
\end{align*} and
\begin{align*}R_x^-&=\{\langle x;x+(0,1)\rangle\mbox{ and }\langle x+(-k,1);x+(-k,2)\rangle\mbox{ are open}\}\\
&\cap\left(\cup_{n=1}^M \{\langle x+(0,1);x+(n,1)\rangle\mbox{ and }\langle x+(n,1);x+(-k,1)\rangle\mbox{ are open}\}\right).
\end{align*}

Observe that the events $A_.(L)$ and $R_.^{\pm}$ use only bonds whose length is at most $k+M$, then taking $K=k+M$, it follows that
\begin{align}\nonumber P^K(R_x^{\pm})&=\delta^2[1-\prod_{n=1}^M (1-p_np_{n+k})]\\
\label{PdeR}&\geq \delta^2[1-\exp(-\sum_{n=1}^M p_np_{n+k})]\geq \delta^2(1-e^{-\eta})
\end{align} where in the last inequality we use (\ref{keM}). 

Finally, we define the event $T_x$ (see Figure \ref{fig:T2}) as follows
$$T_x=A_x(L)\cap (\cup_{i=0}^{\ell-1} R^+_{x+(i,0)})\cap (\cup_{i=\ell+1}^{2\ell} R^-_{x+(i,0)}),$$
then
\begin{align}\label{Red}\nonumber P^K(T_x)&\geq P^K(A_x(L))\cdot P^K(\cup_{i=0}^{\ell-1} R^+_{x+(i,0)})\cdot P^K(\cup_{i=\ell+1}^{2\ell} R^-_{x+(i,0)})\\
\nonumber &\geq (1-\frac{\epsilon}{3})\cdot [1-P^K(\cap_{i=0}^{\ell-1} (R^+_{x+(i,0)})^c)]\cdot [1-P^K(\cap_{i=\ell+1}^{2\ell} (R^-_{x+(i,0)})^c)]\\
\nonumber &\geq (1-\frac{\epsilon}{3})\cdot [1-(1-P^K(R^+_{x}))^{\ell}]^2\\
&\geq (1-\frac{\epsilon}{3})\cdot \left[1-[1-\delta^2(1-e^{-\eta})]^{\ell}\right]^2>1-\epsilon
\end{align}
where in the expression above we are using FKG inequality, the independence of $(R^{\pm}_{x+(i,0)})_i$ and (\ref{l}), respectively.

\begin{figure}[t]
\centering
\includegraphics[width=11cm]{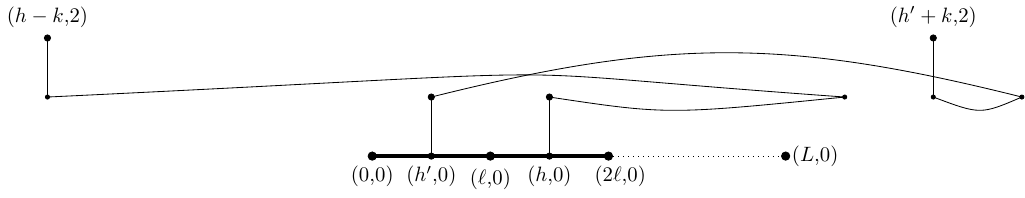}
\caption{The event $T_{(0,0)}$ in the graph $G^{an}$. The thick segment indicates that the vertices therein are connected in the interval $\{0,\dots,L\}\times\{0\}$.}
\label{fig:T2}
\end{figure}

We will construct a site percolation model on the first quadrant of the square lattice $\mathbb{L}^2$. For each site $(v_1,v_2)\in\Z^2_+$, we declare the vertex $(v_1,v_2)$ as {\em red} if and only if the event $T_{(k(v_1-v_2),2(v_1+v_2))}$ occurs. The choice of $k>2L$ and the definition of $T_x$ ensures that all sites are red independently; observe that the path in the event $R^+_x$ ($R^-_x$) starts in the left (respectively right) half of the segment $x+ \{0,\dots,2\ell\}\times\{0\}$. By construction, an infinite path of red sites starting from the origin in $\mathbb{L}^2$ implies in an infinite path of open bonds starting from the origin in $G^{an}$. By (\ref{Red}), each site is red with probability at least $1-\epsilon$; thus ${\displaystyle \lim_{K\to\infty}P^K\{(0,0) \leftrightarrow \infty \} = 1}.$

If $\gcd \{n;p_n>0\}=d>1$, the same proof can be done, with minor modifications, replacing the vertex set $\Z^2$ by $d\Z\times\Z$.
\end{proof}

\begin{remark} The hypotheses of Theorems~\ref{class1} and~\ref{class2} can look strange at first glance. It is an exercise to see that any of these hypotheses are implied by $\sum_n p_n^2 = \infty$, but it is not true the reciprocal affirmation. In the next section, we will give an affirmative answer for the truncation question in an oriented graph under the stronger hypothesis $\sum_n p_n^2 = \infty$.
\end{remark}

\begin{remark}
We finish this section giving examples of sequences where the hypothesis of Theorem~\ref{class1} holds but not that of Theorem~\ref{class2} and vice-versa. Consider the sequences:
\begin{align*}
&p_{n} =
\begin{cases}
k^{-\frac{1}{2}},&\mbox{  if } n=3^k,3^k+1 \mbox{ for some }k,\\
0, &\mbox{   otherwise},
\end{cases}
\end{align*}
and
\begin{align*}
&q_{n}=\begin{cases}
\frac{1}{2\sqrt{k-1}},&\mbox{  if } n\in\{100^k +t3^k;t=1,\dots,k\} \mbox{ for some } k,\\
0,& \mbox{  otherwise.}
\end{cases}
\end{align*}
The sequence $(p_n)_n$ satisfies the hypothesis of Theorem~\ref{class1} but not that of Theorem~\ref{class2}, whilst the opposite situation occurs for the sequence $(q_n)_n$.
\end{remark}

\section{Truncation question on an oriented graph}\label{orient}

\noindent

Let us consider the oriented graph ${\cal G}= (\mathbb{V}({\cal G}),\mathbb{E}({\cal G}))$. The vertex set is $\mathbb{V}({\cal G})=\mathbb{Z}^d\times\mathbb{Z}_+$, elements of $\mathbb{V}({\cal G})$ will be denoted $(x,m)$, where $x \in \mathbb{Z}^d$ and $m \in \mathbb{Z}_+$. The set $\mathbb{E}(\mathcal{G})$ of oriented bonds is
\begin{equation}\{\langle (x,m),(x+n\cdot \vec{e}_i,m+1)\rangle: x \in\mathbb{Z}^d,\;m\in\mathbb{Z}_+,\;i\in\{1,\ldots, d\},\; n\in\mathbb{Z}\}.\label{eq:bonds}\end{equation} 

Again, given a sequence $(p_n)_{n}$ satisfying $\sum_n p_n=\infty$, assume each bond  $\langle (x,m),(x+n\cdot \vec{e}_i,m+1)\rangle$ is open with probability $p_{|n|}$ independently of each other and let $P$ and $P^K$ be the non-truncated and truncated at $K$ probability measures. The event $\{(0, 0)\rightsquigarrow \infty\}$ means that there exists an infinite open oriented path starting from $(0,0)$. 

It was proven in \cite{ELV}, under the hypothesis $\sum_n p_n=\infty$, that $$\lim_{K\rightarrow\infty} P^K\{(0, 0)\rightsquigarrow \infty\}=1$$ for all $d\geq 2$. The case $d=1$ is an open question and a partial answer was given in \cite{AHLV}, more precisely $\lim_{K\rightarrow\infty} P^K\{(0, 0)\rightsquigarrow \infty\}=1$ holds in $d=1$ if $\limsup_{n\rightarrow\infty} p_n>0$. The next theorem improves the result of Theorem 1 of \cite{AHLV} replacing the hypothesis $\limsup_{n\rightarrow\infty} p_n>0$ by $\sum_n p_n^2=\infty$ (that is, some sequences $(p_n)_{n}$ decaying to zero are allowed like $p_n=1/\sqrt{n}$).

\begin{theorem}\label{mainor} For the graph ${\cal G}$ with $d=1$, if  the sequence $(p_n)_n$ satisfies $\sum_n p_n^2=\infty$, the truncation question has an affirmative answer. Moreover, $${\displaystyle \lim_{K\to\infty}P^K\{(0,0) \rightsquigarrow \infty \} = 1}.$$
\end{theorem}
\begin{proof}
This proof is similar to the proof of Theorem 1 of \cite{AHLV}. It consists in to define a family of special events, showing that they induce a supercritical oriented percolation process on an appropriate renormalized lattice, isomorphic to a subset of $\Z^2_+$.

Our first goal is to define the family of events $T^+_{(x,m)}$ and $T^-_{(x,m)}$ for all $(x,m)\in\Z\times\Z_+$. Define $k=\min\{n\in\N;p_n>0\}$; given any $\epsilon>0$ define large enough integers $M$ and $K$ such that
\begin{equation}\label{M}
(1-p_k^2)^M<\epsilon/3\ ,
\end{equation}
\begin{equation}\label{K}
1-\exp\left[-\sum_{i=k+1}^Kp_i^2\right]\geq \left(1-\frac{\epsilon}{3}\right)^\frac{1}{M+1}.
\end{equation}   

Given a vertex $(x,m)\in\Z\times\Z_+$ and $i\in\Z_+$, we define the following events:
$$R^+_{(x,m)}(i)=\{\langle(x,m);(x+i,m+1)\rangle\mbox{ and }\langle(x+i,m+1);(x,m+2)\rangle\mbox{ are open}\},$$
$$R^-_{(x,m)}(i)=\{\langle(x,m);(x-i,m+1)\rangle\mbox{ and }\langle(x-i,m+1);(x,m+2)\rangle\mbox{ are open}\},$$
$$S^+_{(x,m)}=\cup_{i=k+1}^K R^+_{(x,m)}(i),\quad S^-_{(x,m)}=\cup_{i=k+1}^K R^-_{(x,m)}(i)$$
and
$$L_{(x,m)}=\{\langle(x,m);(x+k,m+1)\rangle\mbox{ and }\langle(x+k,m+1);(x+2k,m+2)\rangle\mbox{ are open}\}.$$
Observe that $P(L_{(x,m)})=p_k^2$; since $(R^\pm_{(x,m)}(i))_i$ are independent events, we have that
$$P(S^\pm_{(x,m)})=1-P\left(\cap_{i=k+1}^K (R^\pm_{(x,m)}(i))^c\right)$$
\begin{equation}\label{S}
=\ \ 1-\prod_{i=k+1}^K(1-p_i^2)\geq 1-\exp\left[-\sum_{i=k+1}^Kp_i^2\right]\geq \left(1-\frac{\epsilon}{3}\right)^\frac{1}{M+1},
\end{equation}
where in the last inequality we used (\ref{K}). Now, define our key events $T_{(x,m)}^+$ and $T_{(x,m)}^-$ as
$$T_{(x,m)}^\pm=\left(\cap_{i=0}^M S^\pm_{(x,m+2i)}\right)\bigcap\left(\cap_{i=0}^M S^\pm_{(x+2k,m+2i)}\right)\bigcap\left(\cup_{i=0}^{M-1} L_{(x,m+2i)}\right).$$

Observing that the events $(S^\pm_{(x,m+2i)})_i,\ (S^\pm_{(x+2k,m+2i)})_i$ and $(L_{(x,m+2i)})_i$ are independent, we have by (\ref{M}) and (\ref{S}) that
\begin{align*}
 P(T_{(x,m)}^{\pm})&=\left(\prod_{i=0}^M P(S_{(x,m+2i)}^{\pm})\right)\cdot\left(\prod_{i=0}^M P(S_{(x+2k,m+2i)}^{\pm})\right)\cdot\left(1-P(\cap_{i=0}^{M-1} (L_{(x,m+2i)})^c)\right)\\
       &\geq \left(1-\frac{\epsilon}{3}\right)^2\left(1-(1-p_k^2)^{M}\right)\geq \left(1-\frac{\epsilon}{3}\right)^3\geq 1-\epsilon.
   \end{align*}
	
\begin{figure}[t]
\centering
\includegraphics[width=7cm]{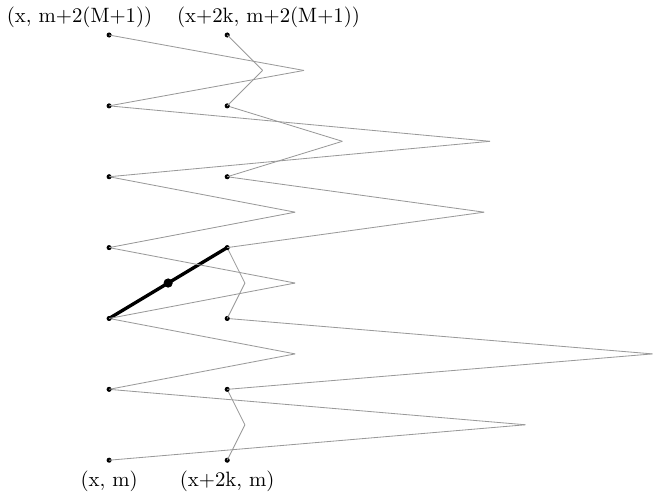}
\caption{The event $T^+_{(x,m)}$ in the graph ${\cal G}$. The edges in bold correspond to the event $L_.$.}
\label{fig:T+}
\end{figure}
It is important to see that all bonds used to define the events $T_{(x,m)}^{\pm}$ have length at most $K$, then $P^K(T_{(x,m)}^{\pm})\geq 1-\epsilon$ also for the truncated measure $P^K$ and furthermore 

\begin{align}\label{T}
       T_{(x,m)}^{\pm}& \subset \{(x,m)\rightsquigarrow (x,m+2(M+1))\}\cap\{(x,m)\rightsquigarrow (x+2k,m+2(M+1))\} . 
   \end{align} 
See Figure \ref{fig:T+} for an illustration of the event $T^+_{(x,m)}$.

Now, define a renormalized graph $G^*=(\mathbb{V}^*,\mathbb{E}^*)$ (an oriented graph), where $\mathbb{V}^*=\{(v,u)\in\mathbb{Z}^2_+; v\leq u\}$ and  $\mathbb{E}^*$ is the set of oriented edges $\mathbb{E}^*=\{\langle(v,u),(w,u+1)\rangle;w=v\mbox{ or }w=v+1\}$. We define each vertex $(v,u)\in \mathbb{V}^*$ of the renormalized as \emph{open} following the rule

\begin{equation*}
\begin{aligned}[c]
 \{(v,u)\text{ is open}\}=\begin{cases}
       T_{(2kv,2(M+1)u)}^+, &\text{ if }  v \text{ is even},\\
       T_{(2kv,2(M+1)u)}^-, &\text{ if }  v \text{ is odd}, 
       \end{cases}
\end{aligned}
\end{equation*} 
and $\{(v,u)\text{ is closed}\}$ otherwise. This appropriate choice of $T^+_.$ or $T^-_.$ holds that the events $(\{(v,u)\text{ is open}\})_{(v,u)}$ are independent, since the set of edges checked for each of these events are disjoint.

Hence,
\begin{equation}\label{bom}P^K\left((v,u)\text{ is open}\right)= P^K(T_{(2kv,2(M+1)u)}^{\pm})>1-\epsilon.
\end{equation}
Furthermore, by (\ref{T})
\begin{align*}
       \left((0,0)\rightsquigarrow (v,u)\right) &\subset \{(0,0)\rightsquigarrow (2kv,2(M+1)(u+1))\}\\
			& \cap\{(0,0)\rightsquigarrow (2k(v+1),2(M+1)(u+1))\} . 
   \end{align*} 

Thus, the cluster of the origin in ${\cal G}$ dominates the oriented site percolation on $G^*$ with parameter $1-\epsilon$. 

Then, we can conclude that 
\[
\displaystyle \lim_{K\to\infty}P^K\{(0,0) \rightsquigarrow \infty \} = 1.
\]
\end{proof}



\noindent

\section*{Acknowledgements}
B.N.B.L. would like to thank and dedicate this paper to his dear friend and PhD advisor Vladas Sidoravicius who told him this fascinating problem a long time ago. Both authors thank Daniel Ungaretti for his valuable comments. The research of B.N.B.L.\ was supported in part by CNPq grant 305811/2018-5 and FAPERJ (Pronex E-26/010.001269/2016). A.M.C. was supported by CNPq. The authors also thank the anonymous referee for his valuable comments.

\end{document}